\documentclass[reqno]{amsart}
\usepackage{amsfonts,amsmath,amssymb}
\usepackage{color}
\usepackage{textcomp}
\usepackage{adjustbox}
\usepackage[T1]{fontenc}
\usepackage[utf8]{inputenc}
\usepackage{lmodern}
\usepackage[backref=page, colorlinks=true, citecolor=cyan, linkcolor=blue]{hyperref}

\numberwithin{equation}{section}

\newtheorem{thm}{Theorem}[section]
\newtheorem{defi}{Definition}[section]
\newtheorem{prop}{Proposition}[section]
\newtheorem{lem}{Lemma}[section]

\newtheorem{rem}{Remark}[section]

\newcommand{\R}{\mathbb{R}} 
\newcommand{\N}{\mathbb{N}}

\newcommand{\xa}{x^{(a)}}
\newcommand{\xb}{x^{(b)}}
\newcommand{\xc}{x^{(c)}}
\newcommand{\xd}{x^{(d)}}
\newcommand{\va}{v^{(a)}}
\newcommand{\vb}{v^{(b)}}
\newcommand{\vc}{v^{(c)}}
\newcommand{\vd}{v^{(d)}}
\newcommand{\Rv}{\tau R^0_V}
\newcommand{\ci}{\left\langle}
\newcommand{\cd}{\right\rangle}
\newcommand{\vu}{\mathbf{v}}

\newcommand{\D}{\mathcal{D}}
\newcommand{\dabv}{d^{(ab)\vu}_V}

\title{The Cucker-Smale model with time delay}

\author{Mauro Rodriguez Cartabia} 
\address{IMAS (UBA-CONICET) and Departamento de Matem\'atica, Facultad de Ciencias Exactas y Naturales, Universidad de Buenos Aires, Ciudad Universitaria, 1428 Buenos Aires, Argentina.}
\email{mrodriguezcartabia@gmail.com, mrodriguezcartabia@dm.uba.ar}
\thanks{The author would like to thank P. Amster for his valuable remarks on this manuscript.} 

\begin{document}


\begin{abstract}
  We study  the classical Cucker-Smale model in continuous time  with a positive time delay $\tau$. As in the   non-delayed case, unconditional flocking occurs when  $\beta \leq 1/2$ for every $\tau>0$.  Furthermore, we prove the exponential decay   for the diameter of the velocities.
\end{abstract}
 
 \maketitle
 
\section{Introduction}

In this paper we extend the original model of Cucker-Smale for continuous time prensent in the article \cite{cucker2007emergent}. It provides a modelization of reaching consensus without a central direction. This study can be  applied to model the collective animal behavior and it is an important part of the fields of developmental biology, neuroscience,
behavioral ecology, sociology \cite{sumpter2010collective},  bacterial colonies \cite{vicsek1995novel}, unmanned vehicles \cite{chuang2007multi}.
 
Given a finite set   $A$  of particles  (we keep this denomination throughout this article, but they can be birds, fishes, robots, etc.)  we reserve the capital letter $N$ to denote the size, this is,  $\# A=N\geq 2$. For  $a \in A$ we consider   the pair    $(\xa(t),\va(t))\in \R^{2d}$  that describes the position and velocity at time $t$, respectively. For physical reasons  we might work in the  three dimensional space setting $d=3$ (also in the two-dimensional setting  $d=2$), but no assumption is needed so we work with arbitrary $d$.  The interactions are modeled by 
$$\frac{d}{dt} v^{(a)}(t)=\sum_{b\in A } \tilde H^{(ab)}(t) \left( \vb(t)-\va(t)\right)$$
where  $\tilde H^{(ab)}(t)$ is  the influence function between  $a$ and  $b$ and depends on the positions   $\xa(t)$ and  $\xb(t)$. In the classical case it is defined by
\begin{equation}\label{ecuacionCS}
 \tilde  H^{(ab)}(t):=\frac{\tilde K}{\left(\sigma^2+\left[\xa(t)-\xb(t)\right]^2\right)^\beta}
\end{equation}
with $\tilde K>0$, $\sigma>0$ and $\beta\geq 0$   fixed.

We extend this model for a fixed delay $\tau>0$ by considering, for each $a\in A$, the system
\begin{equation}\label{sistemaprincipal}
  \left\{\begin{aligned}
    x^{(a)}(t)&=x^{(a)}_0(t)  & \text{ if } t\in[-\tau,0],\\
    v^{(a)}(t)&=v^{(a)}_0(t)  & \text{ if } t\in[-\tau,0],\\
    \frac{d}{dt}x^{(a)}(t)&=v^{(a)}(t) & \text{ if } t\in(-\tau,\infty),\\
    \frac{d}{dt} v^{(a)}(t)&=\sum_{b:b\neq a} H^{(ab)}(t) \left( \vb(t-\tau)-\va(t)\right) &\text{ if } t\in(0,\infty) 
    \end{aligned}
  \right.
\end{equation}
where  
\begin{equation}\label{definicionH}
  H^{(ab)}(t):=\frac{1}{N-1}\psi\left( \left| x^{(a)}(t) -\xb(t-\tau)\right| \right), 
\end{equation}
and $\psi$ is a Lipschitz, non-negative and non-increasing function.  In this system and throughout   this article we adopt the convention that the indices $a,b,c,d$ in a summation are always understood to range over the set of labels $A$. Here $x^{(a)}_0$ and $v^{(a)}_0$ are the initial conditions which in the case of time delay must be time functions defined on the interval $[-\tau,0]$. Without loss of generality, it may be assumed that $(d/dt)\xa_0=\va_0$ for each $a\in A$, but we do not impose this   restriction.

In the following definition, as usual, we denote by  $d_X$ and $d_V$ the  diameters  that measure the maximum distance of positions and velocities between particles ($\Gamma$ and $\Lambda$, respectively, in the original paper  \cite{cucker2007emergent}).

\begin{defi}\label{definicionesd}
We define
\begin{equation*}
  \begin{split}
  d_X(t)&:=\max_{a, b\in A }\left|\xa(t)-\xb(t)\right|,\\
  d_V(t)&:=\max_{a, b\in A }\left|\va(t)-\vb(t)\right|.
  \end{split}
\end{equation*}
We say that a solution $\left(\xa,\va\right)_{a\in A}$ of \eqref{sistemaprincipal}  has  \emph{asymptotic flocking} if
\begin{equation}
\sup_{t\geq 0}d_X(t)<\infty \quad \text{ and }\quad \lim_{t\to\infty}d_V(t)=0.
\end{equation}
\end{defi}

We remark that our model follows the outline of \cite{choi2016cucker}, where normalized communication weights are employed. A similar extension of the Cucker-Smale model was introduced in 
\cite{choi2018emergent}, where the authors proved flocking for   specific initial conditions and delayed times.  We may observe that system (\ref{sistemaprincipal}) can be regarded as a natural way to incorporate a time delay, 
in the sense that each  particle `knows' its position and velocity instantly, but perceives the others at a retarded time.   There are also other works that incorporate time delay in differents forms (see \cite{erban2016cucker,   pignotti2017convergence}).  

Theorem \ref{maintheorem}   is the main result of this article and, to the best of our knowledge,  is the first one that fully extends  the result for the non-delayed case to  the system \eqref{sistemaprincipal}.  It proves that if the integral of the Lipschitz function $\psi$  (see \eqref{definicionH}) diverges then there is an asymptotic flocking independently of the initial conditions. This situation is called \textit{unconditional flocking} and it is worth mentioning that our result does not require conditions on the delay. We use   $\psi$  because it is more general than \eqref{ecuacionCS}.  In the particular case of using the latter  the theorem proves unconditional flocking if $\beta\leq 1/2$ (see Remark \ref{cuckersmale} to deepen this situation) and we show in Section \ref{ejemplos}  that there is no unconditional flocking when $\beta>1/2$.  This is exactly the same result that is valid for the non-delayed case where in \cite{cucker2007emergent} the authors have proved unconditional flocking if $\beta<1/2$ and that it does not occur when $\beta>1/2$. Also in  \cite[Proposition 4.3]{ha2009simple}  unconditional flocking is extended to the case $\beta=1/2$.

In order to find an upper bound for $d_X$, as in \cite{choi2016cucker}, we shall follow the ideas developed in  \cite{ha2009simple}. However, unlike that model,  in ours is not true that  
\begin{equation}\label{novalen}
  \begin{split}
  \frac{d}{dt}d_V(t)&\leq \alpha(1-\delta)d_V(t-\tau)-\alpha d_V(t), \\
  \text{neither}\qquad d_V(t)&\leq \max_{s\in[-\tau,0]}d_V(s) e^{-\tilde Ct},
  \end{split}
\end{equation}
for some positive  constants $\alpha$, $\delta$ and $\tilde C$ with $t\geq 0$ (to clarify this issue see Figure \ref{figura1} and Example \ref{ejemplo1} below). 

The paper is organized as follows. In Section  \ref{seccionteorema} we state the main theorem of this work and dedicate  Section \ref{mainproof} to prove this. Finally, in Section \ref{ejemplos} we study the case when  $\beta>1/2$ and give examples and simulations to illustrate the problem and the issues mentioned in \eqref{novalen}.
                     
\section{Main theorem}\label{seccionteorema}
 
\begin{thm}\label{maintheorem}
Given $\tau>0$  there is  a global solution in time for the system \eqref{sistemaprincipal}. If the function $\psi$ defined in \eqref{definicionH} satisfies that
\begin{equation}\label{hipotesisteorema}
  \int_0^\infty\psi(t)\, dt=\infty
\end{equation}
then there exists a constants $d^*\geq 0$ such that  
$$\sup_{t\geq -\tau}d_X(t)\leq d^*$$
and a constant $C>0$ (that depends on $\psi$,   $\tau$ and the initial condition and is independent of $t$ and $N$, see   \eqref{definicionC}) such that 
\begin{equation}\label{decrecimientoexponencial}
  d_V(t)\leq   \left(  \max_{a,\,b\in A}\max_{s,\,t\in[-\tau,0]} \left|\va(s)-\vb(t)\right|\right) e^{-C(t-2\tau)}  
\end{equation}
for all $t\geq -\tau$.
\end{thm}

\begin{rem}\label{cuckersmale} Recall the influence function given in \eqref{ecuacionCS}. In  the non-delayed case there is unconditional flocking if and only if $\beta\leq 1/2$. With   hypothesis \eqref{hipotesisteorema} we extend this result to the case with time delay since
$$\beta\leq \frac12 \qquad \text{ if and only if } \qquad \int_0^\infty\frac{\tilde K}{\left(\sigma^2+t^2\right)^\beta}\, dt=\infty $$
and in    Proposition  \ref{dosparticulas} we show there is no  unconditional flocking when $\beta>1/2$.
\end{rem}

\section{Proof of   theorem \ref{maintheorem}}\label{mainproof}

To begin with, we   establish in Subsection \ref{preliminares} several definitions we use in the rest of this work and classical results of existence of (local) solution. Then in Subsection \ref{lemasauxiliares} we announce and prove the auxiliary lemmas, including the global existence of solution. Finally,  in Subsection   \ref{flocking} we prove   theorem \ref{maintheorem}.

\subsection{Preliminaries}\label{preliminares}

\begin{defi}\label{definiciones} We recall the definitions of $d_X$ and $d_V$ given in Definition \ref{definicionesd}. Given two particles $a,\,b\in A$, a  vector $\vu\in\R^d$ and  for each $n\in\N_0$, we   define
\begin{equation*}
  \begin{split}
  I_n&:=\max_{c,\, d\in A}\quad  \max_{s,t\in[n\tau-\tau,n\tau]}\left|\vc(s)-\vd(t)\right|,\\
  K&:= \psi(0),\\
  R^0_V&:=\max_{c\in A}\max_{s\in[-\tau,0]}|\vc(s)|,\\
  \dabv(t)&:=\ci\va(t)-\vb(t),\vu\cd,\\
  \phi(t)&:=\min\left\{ e^{-K\tau} \psi\left(\Rv+\max_{s\in[0,t]}d_X(s)\right), \frac{e^{-2K\tau}}{\tau}\right\}
  \end{split}
\end{equation*}
for $t\geq 0$, where $\ci \cdot,\cdot\cd$ is the usual inner product.
\end{defi}
Note that  $I_n$   measures the diameter of the velocities in the entire interval  \linebreak  $ [n\tau-\tau,n\tau]$ of time and that $I_0$ is one of the constants in \eqref{decrecimientoexponencial}. 

Next we prove the existence of solution for the system \eqref{sistemaprincipal} on the interval of time $[0,\tau]$,  although it is classical   we provide the details here for the convenience of readers (we refer to Chapter 3 of \cite{smith2011introduction}, see also \cite{hale1977theory}).

\begin{thm}\label{existencialocal}  
The system \eqref{sistemaprincipal} has a unique  solution on $[0,\tau]$ which is $\mathcal{C}^1$ on $(0,\tau)$.
\end{thm}

\begin{proof} For a fixed $a\in A$, we consider the system
\begin{equation}\label{soluciona}
  \left\{\begin{aligned}
  \left(\xb(t),\vb(t)\right )&=\left(x^{(b)}_0(t),v^{(b)}_0(t) \right) & \text{ if } t\in[-\tau,0], \,b\in A\\
  \frac{d}{dt}\left(\xa(t),\va(t)\right)&= f\left(t,\xa(t),\va(t)\right)& \text{ if } t\in(0,\tau],
  \end{aligned}
  \right.
\end{equation}
where
$$f\left(t,\xa(t),\va(t)\right):=\left(\va(t),\sum_{b:b\neq a} H^{(ab)}(t) \left( \vb(t-\tau)-\va(t)\right)\right).$$ 
Recall Definition \eqref{definicionH} and note that here all functions $(\xb,\vb)$, with $b\neq a$, are defined and depend in time   on the interval $[-\tau,0]$. We want to show the existence of solution for the interval $[0,\tau]$. Therefore we claim that $f(t,x,v)$ is uniformly Lipschitz in $(x,v)$ and continuous in time. By  definition, we have that 
$$\psi\left(\left|x-\xb(t-\tau)\right|\right)\leq \psi(0)=K $$
and that 
$$\left|\vb(t-\tau)\right|\leq R_V^0$$
for each $b\in A$ and $t\in[ 0,\tau]$. Then, for $t\in[0,\tau]$, 
\begin{align*}
  &\left| \psi\left(\left|x-\xb(t-\tau)\right|\right) \left( \vb(t-\tau)-v \right)-\psi\left(\left|y-\xb(t-\tau)\right|\right) \left( \vb(t-\tau)-w \right)\right|\\
  &\qquad\leq \left|\vb(t-\tau)\right|\left| \psi\left(\left|x-\xb(t-\tau)\right|\right) -\psi\left(\left|y-\xb(t-\tau)\right|\right)  \right| +K\left| v-w\right|\\
  &\qquad\leq R_V^0\left|\text{Lip}(\psi)  \left(\left|x-\xb(t-\tau)\right|-\left|y-\xb(t-\tau)\right|\right)  \right| +K\left| v-w\right|\\
  &\qquad\leq R_V^0 \text{Lip}(\psi)  \left|x-y \right|    +K\left| v-w\right|
\end{align*}
where we are denoting by  Lip$(\psi)$   the  Lipschitz constant of $\psi$ on $[0,\tau]$.  Therefore
\begin{align*}
  \left|f(t,x,v)-f(t,y,w)\right|&\leq |v-w|\\
  &\quad+\sum_{b:b\neq a}\frac{1}{N-1}\left| \psi\left(\left|x-\xb(t-\tau)\right|\right) \left( \vb(t-\tau)-v \right)\right.\\
  &\hspace{2cm} \left.-\psi\left(\left|y-\xb(t-\tau)\right|\right) \left( \vb(t-\tau)-w \right)\right|\\
  &\leq  \left(1+\frac{K}{N-1}\right) |v-w|+\frac{R_V^0  \text{Lip}(\psi)}{N-1}|x-y|,\\
  &\leq L|(x,v)-(y,w)|
\end{align*}
if we denote 
\begin{equation*}
  L:= 2\max\left\{\left(1+\frac{K}{N-1}\right),\frac{R_V^0  \text{Lip}(\psi)}{N-1} \right\}.
\end{equation*}
Note that $L$ is independent of time, which means that $f$ is uniformly Lipschitz in $(x,v)$, and is continuous in time because $H^{(ab)}(t)$ is upper bounded and all functions $(\xb,\vb)$ are continuous in time. Then by the Picard-Lindel\"of Theorem (or Cauchy-Lipschitz Theorem)  there exists a unique  $(\xa,\va)$   solution to \eqref{soluciona}. Repeating this argument we get a solution for   all $a\in A$ and we obtain a unique solution for the system \eqref{sistemaprincipal} on the interval $[0,\tau]$ which is  $\mathcal{C}^1$ on $(0,\tau)$.

\end{proof}

\subsection{Auxiliary lemmas}\label{lemasauxiliares}

Although the  following lemma is  similar   to     \cite[Lemma 2.1]{choi2016cucker} and    \cite[Lemma 2.1]{choi2018emergent}, we provide the details here for the convenience of reader.
 
\begin{lem}\label{primerlema} There exists a unique global  solution for the system \eqref{sistemaprincipal} which is $\mathcal{C}^1$ for positive times. Also for each  $n\in \N_0$, vector $\vu$ and $a\in A$,   we have   that 
\begin{equation}\label{cotavelocidad}
\min_{c\in A}\min_{s\in[n\tau-\tau,n\tau]}\ci\vc(s),\vu\cd\leq \ci\va(t),\vu\cd\leq \max_{c\in A}\max_{s\in[n\tau-\tau,n\tau]}\ci\vc(s),\vu\cd 
\end{equation}
for all $t\geq (n\tau-\tau)$. In particular, we have that $I_n\geq I_{n+1}.$ 

As well, for each $a,\,b\in A$, we get that
\begin{equation}\label{cotaRV0}
  |\va(t)|\leq R_V^0
\end{equation}  for all $t\geq -\tau$ and  
\begin{equation}\label{cotadX}
  \left|\xa(t-\tau)-\xb(t)\right| \leq  \Rv+d_X(t-\tau)
\end{equation}
for all $t\geq 0$.
\end{lem}

\begin{proof}
To begin with, fix a vector $\vu$. We want to prove the inequality \eqref{cotavelocidad} for the case $n=0$. We proceed by contradiction. Suppose  there are $t_0\in [0 , \tau]$ and  $a\in A$ such that 
$$M=\max_{c\in A}\max_{ s\in[ -\tau, 0]}\ci\vc(s),\vu\cd<\ci\va(t_0),\vu\cd.$$
Since $\ci\va(0),\vu\cd\leq M $, $\va$ is smooth on $(0,\tau)$ and $t_0\neq 0$, by the mean value theorem, there exists $t_1\in (0 ,t_0)$ such that  
 \begin{align*}
   M&<\ci\va(t_1),\vu\cd,\\
   0&<\frac{d}{dt}\ci\va(t_1),\vu\cd.
 \end{align*}
In particular  $(t_1-\tau)\in [ -\tau,0]$, so then
\begin{align*}
  \frac{d}{dt}\ci\va(t_1),\vu\cd&= \sum_{b:b\neq a}H^{(ab)}(t_1)\ci \vb(t_1-\tau)-\va(t_1),\vu\cd\\
  &\leq \sum_{b\neq a}H^{(ab)}(t_1)\left(M -\ci  \va(t_1),\vu\cd\right)\\
  &\leq 0
\end{align*}
and  we get the contradiction. For the case of the minimum, apply the same reasoning to the vector $(-\vu)$.

To continue, we prove   inequality \eqref{cotaRV0} for the case $t\in[0,\tau]$. Fix $a\in A$ and $t_0\in[0,\tau]$, if  $|\va(t_0)|=0$ it is obvious. In other case, define 
$$\vu=\frac{\va(t_0)}{|\va(t_0)|}.$$
By inequality   \eqref{cotavelocidad} and the fact that $\vu$ is a unit vector, using  Cauchy–Schwarz inequality, we get that
$$|\va(t_0)|=\ci\va(t_0),\vu\cd\leq \max_{b\in A}\max_{s\in[ -\tau,0]}\ci\vb(s ),\vu\cd\leq R_V^0.$$
 
Now we have bounded  $\va(t)$ for each $a\in A$ and $t\in[0,\tau]$. Then  applying Theorem \ref{existencialocal} to the system \eqref{sistemaprincipal} for initial conditions on the interval of time $[0,\tau]$, we get a unique extension of solution to the interval $[\tau,2\tau]$. We  repeat this argument to prove inequality \eqref{cotavelocidad} for  the case $n=1$ and  inequality \eqref{cotaRV0} for the case $t\in[\tau,2\tau]$. We again applying Theorem \ref{existencialocal} to the system \eqref{sistemaprincipal}, but with initial conditions on $[\tau,2\tau]$ and so on. Finally, we get  the inequalities \eqref{cotavelocidad} for each $n\in\N_0$, \eqref{cotaRV0} to all $t\geq -\tau$ and global unique existence of solution. Note that by the uniqueness of solution it is $\mathcal{C}^1$ on $(0,\infty)$. 

To finish the proof, it remains to prove that $I_n\geq I_{n+1} $ and the inequality \eqref{cotadX}. Given  $n\in\N_0$, fix $a,\,b\in A$ and $t_1,t_2\in[n\tau,n\tau+\tau]$ such that 
$$I_{n+1}=\left|\va(t_1)-\vb(t_2)\right|.$$
Again, if $I_{n+1}=0$ then it is obvious, and if $I_{n+1}>0$, define
\begin{align*}
  \vu& =\frac{\va(t_1)-\vb(t_2)}{\left|\va(t_1)-\vb(t_2)\right|} .
\end{align*}
 Then 
\begin{align*}
  I_{n+1}&=\ci \va(t_1)-\vb(t_2),\vu\cd\leq \max_{c,d\in A}\max_{s,t\in[n\tau-\tau,n\tau]}\ci\vc(s)-\vd(t),\vu\cd \leq I_n
\end{align*}
and the first result is proved. For the other, note that 
\begin{align*}
  \left|\xa(t)-\xb(t-\tau)\right|&\leq \left|\xa(t)-\xa(t-\tau)\right|+\left|\xa(t-\tau)-\xb(t-\tau)\right|\\
  &\leq \int_{t-\tau}^t\left|\va(s)\right|\, ds+d_X(t-\tau)\\
  &\leq \int_{t-\tau}^tR_V^0\, ds+d_X(t-\tau)\\
  &\leq \Rv+d_X(t-\tau)
\end{align*}
and we get the lemma.
\end{proof}

\begin{lem}\label{cotasMm}
For all $a,\,b\in A$, unit vector $\vu$ and $n\in\N_0$, we have that
\begin{equation*}
  \begin{split}
  \dabv  (t) &\leq e^{-K(t-t_0)}\dabv  (t_0)+(1-e^{-K(t-t_0)})I_n,\\
  I_{n+1}&\leq e^{-K\tau}d_V(n\tau ) +(1-e^{-K\tau})I_{n },\\
  \end{split}
\end{equation*}
for all $t\geq t_0\geq  n\tau$.
\end{lem}

\begin{proof}
Fix a unit vector $\vu$ and denote
\begin{align*}
  M&=\max_{c\in A}\max_{s\in[n\tau-\tau,n\tau]}\ci \vc(s),\vu\cd,\\
  m&=\min_{c\in A}\min_{s\in[n\tau-\tau,n\tau]}\ci \vc(s),\vu\cd.
\end{align*} 
Note that, by the Cauchy–Schwarz inequality, 
\begin{equation}\label{desigualdadMmI}
  \begin{split}
  M-m&=\max_{c,\,d\in A}\max_{s,\,t\in[n\tau-\tau,n\tau]} \ci \vc(s)-\vd(t),\vu\cd\\
  &\leq \max_{c,\,d\in A}\max_{s,\,t\in[n\tau-\tau,n\tau]} \left|\vc(s)-\vd(t)\right|\\
  &= I_n
  \end{split}
\end{equation}
Now we claim that for each $t\geq t_0\geq  n\tau$, we have that 
\begin{equation}\label{gronwall}
\begin{split}
  \ci\va(t),\vu\cd&\leq e^{-K(t-t_0)}\ci\va(t_0  ),\vu\cd+ \left(1-e^{-K(t-t_0)}\right) M ,\\
  \ci\vb(t),\vu\cd&\geq e^{- K(t-t_0)}\ci\vb(t_0 ),\vu\cd+  \left(1-e^{-K(t-t_0)}\right)m .
\end{split}
\end{equation}
If $t\geq    n\tau  $ then for each $c\in A$  we get that 
$$ \ci\vc(t-\tau)-\va(t),\vu\cd\leq M -\ci \va(t),\vu\cd$$
and the right-hand side is not negative by inequality  \eqref{cotavelocidad} of Lemma \ref{primerlema}. Then
\begin{align*}
  \frac{d}{dt}\ci\va(t),\vu\cd&=\sum_{c:c\neq a} H^{(ac)}(t)  \ci\vc(t-\tau)-\va(t),\vu\cd \\
  &\leq \sum_{c:c\neq a} H^{(ac)}(t)   \left(M -\ci \va(t),\vu\cd\right) \\
  &\leq \sum_{c:c\neq a} \frac{K}{N-1} \left(M -\ci \va(t),\vu\cd \right)\\
  &= K \left(M -\ci \va(t),\vu\cd \right)
\end{align*}
and it is enough to apply  Gr\"onwall's Lemma 
to get the first equation of \eqref{gronwall}. For the other   apply this to   particle $b$ and vector $(-\vu)$. 

To prove the first inequality of the lemma note that 
\begin{align*}
  \dabv(t)&=  \ci\va(t)-\vb(t),\vu\cd\\
  &\leq e^{-K(t-t_0)}\ci\va(t_0  )-\vb(t_0),\vu\cd+ \left(1-e^{-K(t-t_0)}\right) (M-m) ,
\end{align*}
hence it is enough  to apply inequalities \eqref{desigualdadMmI}. 

For the second one fix $a,\,b\in A$ and $t_1,\,t_2\in[n\tau,n\tau+\tau]$ such that
$$I_{n+1}=\left|\va(t_1)-\vb(t_2)\right|.$$
As we did in Lemma \ref{primerlema}, if $I_{n+1}=0$ the result is obvious, and if $I_{n+1}>0$   define
\begin{align*}
  \vu& =\frac{\va(t_1)-\vb(t_2)}{\left|\va(t_1)-\vb(t_2)\right|}.
\end{align*}
Then using inequalities \eqref{gronwall},  with $t_0=n\tau$, we get
\begin{align*}
   \ci\va(t_1),\vu\cd&\leq e^{-K(t-n\tau )}\ci\va(n\tau   ),\vu\cd+ \left(1-e^{-K(t-n\tau )}\right) M\\
   &\leq e^{-K(t-n\tau )}\left(\ci\va(n\tau   ),\vu\cd-M\right)+  M\\
   &\leq e^{-K \tau  }\left(\ci\va(n\tau   ),\vu\cd-M\right)+  M\\
    &\leq e^{-K \tau }\ci\va(n\tau),\vu\cd+ \left(1-e^{-K\tau}\right) M  
\end{align*}
because $M\geq \ci\va(n\tau   ),\vu\cd $ and $(t-n\tau)\leq \tau$. Similarly,
$$\ci\vb( t_2 ),\vu\cd\geq e^{- K\tau}\ci\va(n\tau),\vu\cd+  \left(1-e^{-K\tau}\right)m .$$
Finally, again by the Cauchy–Schwarz inequality, 
$$\dabv(n\tau)\leq d_V(n\tau)$$
and we get the  result  because
\begin{align*}
  I_{n+1}&=\ci\va(t_1)-\vb(t_2),\vu\cd\\
  &\leq e^{-K \tau }\dabv(n\tau)+ \left(1-e^{-K\tau}\right) (M-m).
\end{align*}
\end{proof}

\begin{lem}\label{mainlemma}
For each $n\geq 2$  we have that
\begin{equation*}
  \begin{split}
  I_{n+1}&\leq \left(1- e^{-K\tau}\int_{n\tau-2\tau}^{n\tau-\tau}\phi(s)\,ds\right)I_{n-2}
  \end{split}
\end{equation*}
where $I_n$, $K$ and $\phi$ given are given in Definition \ref{definiciones}.
\end{lem}

\begin{proof}
We sketch the proof in three steps, the first two dedicated to prove that
\begin{eqnarray}\label{desigualdaddV}
  d_V(n\tau)&\leq  \left(1- \int_{n\tau-\tau}^{n\tau}\phi(s)\,ds\right)I_{n-2}
\end{eqnarray}
and the  third to prove the  inequality of the lemma itself. As before, assume $d_V(n\tau)\neq 0$ and fix $a$, $b$ such that  
$$d_V(n\tau)=\left|\va(n\tau)-\vb(n\tau)\right|$$
and define
$$\vu=\frac{\va(n\tau)-\vb(n\tau)}{\left|\va(n\tau)-\vb(n\tau)\right|}.$$
Then 
$$d_V(n\tau)=\ci\va(n\tau)-\vb(n\tau),\vu\cd=\dabv(n\tau).$$ 
We have two cases to study, first we analyze when $\dabv(t_0)<0$ for some time $t_0\in[ n\tau-2\tau,n\tau-\tau]$ and later when $\dabv\geq 0$ for all times in $[ n\tau-2\tau,n\tau-\tau]$.

\textit{Step 1.}  Suppose there exists  $t_0\in[ n\tau-2\tau,n\tau-\tau]$ such that $\dabv(t_0)<0$. By definition of $\phi$ we get 
\begin{equation}\label{primercaso}
  1-\int_{n\tau-\tau}^{n\tau}\phi(s)\,ds \geq 1-\int_{n\tau-\tau}^{n\tau}\frac{e^{-2K\tau}}{\tau}\,ds\geq 1-e^{-2K\tau}.
\end{equation}
Then, by the first inequality of Lemma \ref{cotasMm}, we have that 
\begin{align*}
  \dabv(n\tau)&\leq e^{-K(n\tau-t_0)}\dabv  (t_0)+(1-e^{-K(n\tau-t_0)})I_{n-2},\\
  &\leq (1-e^{-K(n\tau-t_0)})I_{n-2},\\
  &\leq (1-e^{-2K\tau })I_{n-2}
\end{align*}
and by inequality \eqref{primercaso} we get \eqref{desigualdaddV}.

\textit{Step 2.} Now suppose the opposite, that is,  
\begin{equation}\label{segundocaso}
  \dabv(t_0)\geq 0
\end{equation} 
for each $t_0\in[ n\tau-2\tau,n\tau-\tau]$ and denote 
\begin{align*}
  M&=\max_{c\in A}\max_{s\in[n\tau-2\tau,n\tau-\tau]}\ci \vc(s),\vu\cd,\\
  m&=\min_{c\in A}\min_{s\in[n\tau-2\tau,n\tau-\tau]}\ci \vc(s),\vu\cd. 
\end{align*} 

Then, if $t\in [n\tau-\tau,n\tau]$, we get 
\begin{align*}
  \frac{d}{dt}\dabv(t )&= \sum_{c: c\neq a} H^{(ac)}( t)  \ci \vc( t -\tau)-\va( t),\vu\cd\\
  &\quad +\sum_{c: c\neq b} H^{(bc)}( t)  \ci\vb( t  )-\vc( t-\tau),\vu\cd\\
  &= \sum_{ c:c\neq a} H^{(ac)}( t)  \left(\ci \vc( t -\tau),\vu\cd-M+M-\ci\va( t),\vu\cd\right)\\
  &\quad +\sum_{c: c\neq b} H^{(bc)}( t)  \left(\ci\vb( t  ),\vu\cd-m+m-\ci\vc( t-\tau),\vu\cd\right).
\end{align*}

By definition of  $H$  (see Equation \eqref{definicionH}) we have that 
$$\sum_{c,c\neq a} H^{(ac)}( t) \leq \sum_{c,c\neq a} \frac{1}{N-1} \psi(0) \leq K   $$
and, by inequality \eqref{cotadX} of Lemma \ref{primerlema},
\begin{align*}  
  H^{(ac)}( t)&=\frac{1}{N-1}\psi\left(\left|\xa( t)-\xc( t-\tau)\right|\right)\\
  &\geq \frac{1}{N-1}\psi\left(\Rv+d_X(t-\tau)\right)\\
  &\geq \frac{1}{N-1}\psi\left(\Rv+\max_{s\in[0,t]}d_X(s-\tau)\right)\\
  &\geq \frac{ e^{K\tau}\phi(t-\tau)}{N-1} .
\end{align*}
  Therefore
\begin{align*}
  \frac{d}{dt}\dabv(t)&\leq      \sum_{ c:c\neq a ,b   } \frac{ e^{K\tau}\phi(t-\tau)}{N-1} \left( \ci\vc(t-\tau) ,\vu\cd-M\right) \\
  &\quad  + \sum_{c:c\neq a ,b   } \frac{ e^{K\tau}\phi(t-\tau)}{N-1} \left(m- \ci\vc(t-\tau) ,\vu\cd \right)  \\ 
  &\quad +\frac{ e^{K\tau}\phi(t-\tau)}{N-1} \left( \ci\vb(t-\tau) ,\vu\cd-M\right)\\
  &\quad  +\frac{ e^{K\tau}\phi(t-\tau)}{N-1} \left(m- \ci\va(t-\tau) ,\vu\cd \right)\\
  &\quad +K\left(M-\ci\va(t),\vu\cd\right)\\
  &\quad +K\left(\ci\vb(t),\vu\cd-m\right)\\
  &\leq  -  e^{K\tau}\phi(t-\tau)(M-m)+ K\left(M-m-\dabv(t)\right)\\
  &=     \left(K-e^{K\tau}\phi(t-\tau)\right)I_{n-1}-K \dabv(t)
\end{align*}
where we use inequality \eqref{segundocaso}. By Gr\"onwall's Lemma, we get 
\begin{align*}
  \dabv(n\tau)&\leq   e^{-K\tau}\dabv(n\tau-\tau)+I_{n-1}\int_{n\tau-\tau}^{n\tau}e^{-K(n\tau-s)}(K-e^{K\tau}\phi(s-\tau))\, ds\\
  &=  e^{-K\tau}\dabv(n\tau-\tau)\\
  &\quad +I_{n-1}\left(1-e^{-K\tau}-\int_{n\tau-\tau}^{n\tau}e^{-K(n\tau-s)}  e^{K\tau}\phi(s-\tau) \, ds\right)\\
  &\leq  e^{-K\tau}I_{n-1} +I_{n-1}\left(1-e^{-K\tau}-\int_{n\tau-\tau}^{n\tau} \phi(s-\tau) \, ds\right)\\
  &\leq \left(1 -\int_{n\tau-2\tau}^{n\tau-\tau} \phi(s) \, ds\right)I_{n-1}\\
  &\leq \left(1 -\int_{n\tau-2\tau}^{n\tau-\tau} \phi(s) \, ds\right)I_{n-2}
\end{align*}
and the  inequality \eqref{desigualdaddV} is proved.

\textit{Step 3.} 
To finish the proof, we use  Lemma \ref{cotasMm} and inequality \eqref{desigualdaddV} to get 
\begin{align*}
  I_{n+1}&\leq e^{-K\tau}d_V(n\tau ) +(1-e^{-K\tau})I_{n }\\
  &\leq e^{-K\tau}\left(1- \int_{n\tau-2\tau}^{n\tau-\tau}\phi(s)\,ds\right)I_{n-2}+(1-e^{-K\tau})I_{n-2 }\\
  &\leq \left(1- e^{-K\tau}\int_{n\tau-2\tau}^{n\tau-\tau}\phi(s)\,ds\right)I_{n-2}
\end{align*}
and the result is proved.
\end{proof}

\subsection{Proof of Theorem \ref{maintheorem}}\label{flocking}
We need to find an upper bound for $d_X$. To this purpose we define a function $\mathcal{L}$ which implies that 
$$\psi\left(\tau R_V^0+\max_{s\in[0,t]}d_X(s)\right)$$
has an   under bound for all $t\geq-\tau$. Finally, this gives the exponential decay of $d_V$. 

\begin{proof} We  define the function 
$$\D(t):=\left\{\begin{array}{ll}
  I_0&\text{if } t\in[-\tau,2\tau],\\
  \D(n\tau)\left(1- e^{-K\tau}\int_{n\tau}^{t}\phi(s)\,ds\right)^{1/3} &\text{if }t\in(n\tau,n\tau+\tau],\,n\in\N,\,n\geq 2. 
\end{array}\right.
$$

Note that it  is continuous, non-increasing and almost every time differentiable. We sketch the proof in three steps. First we claim that $\D(t)$ is an upper bound for  $I_n$ if $t\in[-\tau,n\tau]$ for all $n\in\N_0$. Then we show that $d_X$ is upper bounded, and finally we obtain the exponential decrease in time for  $d_V$.

\textit{Step 1.} Since $\phi(t)$ is non-increasing we have that
\begin{equation*}
  1- e^{-K\tau}\int_{m\tau}^{m\tau+\tau}\phi(s)\,ds\leq 1- e^{-K\tau}\int_{n\tau}^{n\tau+\tau}\phi(s)\,ds
\end{equation*}
if $m\leq n$ and then
\begin{equation}\label{tercios}
  \left(1- e^{-K\tau}\int_{n\tau-2\tau}^{n\tau-\tau}\phi(s)\,ds\right) \leq \prod_{i=0}^2\left(1- e^{-K\tau}\int_{n\tau-2\tau+i\tau}^{n\tau-\tau+i\tau}\phi(s)\,ds\right)^{1/3}.
\end{equation}
Note that the right-hand side of the inequality above is a telescoping product that equals  exactly
$$\frac{\D(n\tau+\tau)}{\D(n\tau-2\tau)}.$$
 Now we claim that
$$I_{n+1}\leq \D(t)$$
for each $t\in[-\tau,n\tau+\tau]$ and $n\in\N_0$. We proceed by induction on $n$.  If $n\leq 2$, by definition of $\D$ we get 
$$I_2\leq I_1\leq I_0 =\D(t)$$
for each $t\in[-\tau,2\tau]$. Then  suppose   that 
$$I_{n-2}\leq D(n\tau-2\tau) $$
and we want to prove that
$$I_{n+1}\leq D(t) $$
for each $t\in[-\tau,n\tau+\tau].$ By Lemma \ref{mainlemma} and inequality \eqref{tercios},  we have that 
\begin{align*}
  I_{n+1}&\leq  \left(1- e^{-K\tau}\int_{n\tau-2\tau}^{n\tau-\tau}\phi(s)\,ds\right)I_{n-2}\\
  &\leq  \left(1- e^{-K\tau}\int_{n\tau-2\tau}^{n\tau-\tau}\phi(s)\,ds\right)D(n\tau-2\tau)\\
  &\leq \D(n\tau+\tau)\\
  &\leq \D(t)
\end{align*}
for each $t\in[-\tau,n\tau+\tau]$ since $\D$ is a non-increasing function.

\textit{Step 2.} Note that for almost every time
$$\frac{d}{dt}\left(\max_{s\in[0,t]}d_X(s)\right)\leq \left|\frac{d}{dt}d_X(t)\right|\leq d_V(t).$$
For the first inequality, notice  that for almost every time, either  $\max_{s\in[0,t]}d_X(s)$ is constant, or increases as $d_X(t)$. And the second follows by definition of $d_X$ and $d_V$.  Then define the function 
$$\mathcal{L}(t)=\D(t)+\frac{e^{-K\tau}}{3}\int_0^{\Rv+\max_{s\in[0,t]}d_X(s)}\min\left\{ e^{-K\tau} \psi(s),\frac{e^{-2K\tau}}{\tau}\right\}\,ds.$$
For each $n\in\N_0$ and for almost every time $t\in(n\tau,n\tau+\tau)$ we have that 
\begin{align*}
  \frac{d}{dt}\mathcal{L}(t)&=\frac{d}{dt}\D(t)+\frac{e^{-K\tau}}{3}\phi(t)\frac{d}{dt}\left(\Rv+\max_{s\in[0,t]}d_X(s)\right)\\
  &\leq \frac{-e^{-K\tau}}{3} \phi(t)\left(1- e^{-K\tau}\int_{n\tau}^{t}\phi(s)\,ds\right)^{-2/3}\D(n\tau)+\frac{e^{-K\tau}}{3}\phi(t)d_V(t)\\
  &\leq \frac{e^{-K\tau}}{3}\phi(t)\left(d_V(t)-\D(n\tau)\right)\\
  &\leq 0
\end{align*}
where we are    using that 
$$d_V(t)\leq I_{n+1}\leq I_n\leq \D(n\tau).$$
Therefore $\mathcal{L}$ is a  non-increasing function  and because $\D$ is non-negative we get 
\begin{equation}\label{desigualdadD}
  \frac{e^{-K\tau}}{3}\int_0^{\Rv+\max_{s\in[0,\infty)}d_X(s)}\min\left\{ e^{-K\tau} \psi(s),\frac{e^{-2K\tau}}{\tau}\right\}\,ds\leq   \mathcal{L}(0)=I_0. 
\end{equation}
Recall Hypothesis \eqref{hipotesisteorema}, this is, 
$$\int_0^\infty \psi(s)\,ds=\infty$$
which combined with \eqref{desigualdadD} implies there exists  $d^*$ such that
$$\Rv+\max_{s\in[0,\infty)}d_X(s)\leq d^*$$ 
and, therefore, $d_X$ is upper bounded.

\textit{Step 3.} To finish the proof observe that $\psi$ and $\phi$ are positive by Hypothesis \eqref{hipotesisteorema}, and the fact that $d_X$ is upper  bounded implies that $\phi$ has a positive lower bound  that we denote by $\phi(\infty)$, this is,
$$\left(1- e^{-K\tau}\int_{n\tau}^{n\tau+\tau}\phi(s)\,ds\right)^{1/3} \leq \left(1- e^{-K\tau} \phi(\infty)\tau\right)^{1/3}<1$$
for each $n\geq 2$. Then we define
\begin{equation}\label{definicionC}
  C:=\frac{1}{3\tau}\ln\left( \frac{1}{1-e^{-K\tau}\tau\phi(\infty)} \right),
\end{equation}
which is a positive constant, and  it happens that 
$$\D(n\tau)\leq I_0 e^{-C(n-2)\tau}$$
for each $n\geq 2$. Finally, given $t\geq 2\tau$, there is an integer  $m\geq 2$  such that \linebreak $t\in[m\tau,m\tau+\tau]$
and
$$d_V(t)\leq I_{m+1}\leq \D(m\tau+\tau)\leq I_0e^{-C(m-1)\tau}\leq I_0e^{-C(t-2\tau)}$$
which proves the theorem.
\end{proof}

\section{The case of two particles}\label{ejemplos}

In this section we asume that $A=\{a,b\}$ and define
\begin{align*}
  X(t) &:=\xa(t)-\xb(t),\\
  V(t) &:=\va(t)-\vb(t).
\end{align*}

\subsection{The case without flock}

The following proposition adapts the idea in    \cite[section IV]{cucker2007emergent} to prove there is no  unconditional flocking when $\beta>1/2$. Note that this happens for every delayed time  $\tau>0$.

\begin{prop}\label{dosparticulas} Fix $\tau>0$ and suppose 
$$H^{(cd)}(t)= \frac{1}{\left(1+\left[\xc(t)-\xd(t-\tau)\right]^2\right)^\beta}$$
with $\beta>1/2$. Then there exist initial conditions $\left(X(s),V(s)\right)_{s\in[-\tau,0]}$ such that   the system \eqref{sistemaprincipal} has no asymptotic  flocking. 
\end{prop}

\begin{proof}
Define the initial conditions 
\begin{align*}
  \left\{\begin{aligned}
  \xa(s)&=s+\tau^{1/(2\beta)}+2\tau+\left(3\,\frac{2^\beta }{2\beta-1}\right)^{1/(2\beta-1)} & \text{ if } s\in[-\tau,0],\\
  \va(s)&=1 & \text{ if } s\in[-\tau,0],\\
  \xb(s)&= -s & \text{ if } s\in[-\tau,0],\\
  \vb(s)&=-1 & \text{ if } s\in[-\tau,0],
  \end{aligned}
  \right.\\
\end{align*}
and denote
$$\mathcal{S}=\left\{t\geq 0: V(s)\geq 1 \text{ for each } s\in[-\tau,t]\right\}.$$
Because $V(0)=2$ it happens that  $\mathcal{S}$ is not empty. Therefore denote $s^*=\sup\mathcal{S}$ and we claim that $s^*=\infty$. For contradiction assume $s^*<\infty$. Since $V$ is continuous we have that $V(s^*)=1$. We sketch   the proof in two steps. First we prove that 
\begin{equation}\label{Vpositivas}
  \begin{split}
  \va(t-\tau)-\vb(t)\geq 0,\\
  \va(t )-\vb(t-\tau)\geq 0
  \end{split}
\end{equation} 
for all $t\in[-\tau,s^*]$ and in the second step we conclude that $V(s^*)>1$ and get the contradiction.

\textit{Step 1.} Since $V(t)$ is positive for times less than $s^*$ we have that 
\begin{equation}\label{Xcrece}
  X(t)\geq X(-\tau)>\tau^{1/(2\beta)}+\tau
\end{equation}
if $t\in[-\tau,s^*].$ Note that by inequality \eqref{cotadX} of Lemma \ref{primerlema} and the fact that $\Rv= 1$, we have that 
\begin{equation}\label{cotauno}
  |\va|,\,|\vb|\leq 1.
\end{equation}
 Then  for $t\in[0,s^*]$, we get 
\begin{equation}\label{desigualdadX}
  \begin{split}
  \left|\xa(t-\tau)-\xb(t)\right|&\geq  \left|\xa(t )-\xb(t )\right|-\int_{t-\tau}^t \va(s)\,ds\\
  &=X(t ) -\tau,\\
  \end{split}
\end{equation}
and, similarly,
\begin{equation}\label{desigualdadX2}
  \left|\xb(t-\tau)-\xa(t)\right|\geq   X(t ) -\tau.
\end{equation}

 Now, for $t\in[0, s^*]$ and using \eqref{Xcrece}, \eqref{cotauno} and \eqref{desigualdadX}, we have that
\begin{align*}
  \va(t-\tau)-\vb(t)&= \va(t-\tau)-\vb(t-\tau)-\int_{t-\tau}^t \frac{d}{dt}\vb(s)\, ds\\
  &= V(t-\tau) -\int_{t-\tau}^t \frac{\va(s-\tau)-\vb(s)}{\left(1+(\xa(s-\tau)-\xb(s))^2\right)^\beta}\, ds\\
  &\geq V(t-\tau)-\int_{t-\tau}^t \frac{2 }{\left(1+(X(t ) -\tau)^2\right)^\beta}\, ds\\
  &\geq V(t-\tau)-\int_{t-\tau}^t \frac{2 }{\left(1+\tau^{1/\beta}\right)^\beta}\, ds\\
  &\geq V(t-\tau)- \frac{2\tau}{\left(1+\tau^{1/\beta}\right)^\beta}\\
  &> V(-\tau)-2\\
  &\geq 0,
\end{align*}
since $V$ increases. Similarly, we have that 
$$\va(t)-\vb(t-\tau)\geq V(t-\tau )-2> V(-\tau )-2 \geq0 $$
and we get \eqref{Vpositivas}.

\textit{Step 2.} Since  
\begin{equation}\label{desigualdadbeta}
  (1+|X|)^{2\beta}\leq 2^\beta (1+X^2)^\beta
\end{equation}
 using inequalities \eqref{Vpositivas}, we have that 
\begin{align*}
  \frac{d}{dt}V(t)&= -\frac{\va(t)-\vb(t-\tau)}{\left(1+\left(\xa(t)-\xb(t-\tau)\right)^2\right)^{\beta}}-\frac{\va(t-\tau)-\vb(t)}{\left(1+\left(\xb(t)-\xa(t-\tau)\right)^2\right)^{\beta}}\\
  &\geq -2^\beta\left(\frac{\va(t)-\vb(t-\tau)}{\left(1+\left|\xa(t)-\xb(t-\tau)\right|\right)^{2\beta}}+\frac{\va(t-\tau)-\vb(t)}{\left(1+\left|\xb(t)-\xa(t-\tau)\right|\right)^{2\beta}}\right).
\end{align*}
By inequalities \eqref{desigualdadX} and \eqref{desigualdadX2}, we get 
\begin{align*} 
  \frac{d}{dt}V(t) &\geq -2^\beta\left(\frac{\va(t)-\vb(t-\tau)}{\left(1+ X(t) -\tau \right)^{2\beta}}+\frac{\va(t-\tau)-\vb(t)}{\left(1+ X(t) -\tau \right)^{2\beta}}\right)\\
   &\geq -2^\beta \frac{V(t)+V(t-\tau)}{\left(1+X(t) -\tau \right)^{2\beta}}.
\end{align*}
Because $V(t)\geq 1$ in $\mathcal{S}$ we have that 
$$ 2 V(t) \geq 2 =I_0\geq V(t-\tau) $$
where we are using inequality \eqref{cotavelocidad} of Lemma \ref{primerlema} and we get
\begin{align*}
   \frac{d}{dt}V(t)&\geq -2^\beta  \frac{3V(t)}{\left(1+ X(t) -\tau \right)^{2\beta}}  = -2^\beta   \frac{3X'(t)}{\left(1+ X(t)-\tau \right)^{2\beta}}.
\end{align*}
Now integrating  over time and using  definition of $X(0)$, we get  that 
\begin{align*}
  V(s^*)&\geq V(0)+\int_0^{s^*}  2^\beta  \frac{-3X'(s)}{\left(1+ X(s) -\tau \right)^{2\beta}}\, ds\\
  &=2+  \frac{2^\beta}{2\beta-1}  \left( \frac{ 3 }{\left(1+ X(s^*) -\tau \right)^{2\beta-1}}-\frac{3}{\left(1+ X(0) -\tau \right)^{2\beta-1}}\right)\\
  &>2+     \frac{ 2^\beta\cdot 3}{(2\beta-1)\left(1+ X(s^*)-\tau \right)^{2\beta-1}}-\frac{3\frac{2^\beta}{2\beta-1}}{\left(1+ \tau+ \left(3\,\frac{2^\beta }{2\beta-1}\right)^{1/(2\beta-1)} -\tau \right)^{2\beta-1}} \\
  &>2+   \frac{  2^\beta\cdot 3}{(2\beta-1)\left(1+ X(s^*) -\tau \right)^{2\beta-1}}   -1\\
  &>1
\end{align*}
and we get the contradiction. Then $s^*=\infty$ which means that $V(t) >1$ for all $t\geq -\tau$ and therefore there is no asymptotic flocking.
\end{proof}

\subsection{Example 1}\label{ejemplo1}

Fix $\varepsilon$ such that $0<\varepsilon<<\tau$ and suppose
\begin{equation}
  \left\{\begin{aligned}
  \va(s)&=1 & \text{ if } s\in[-\tau,-\varepsilon],\\
  \va(s)&=-s/\varepsilon & \text{ if } s\in[-\varepsilon,0],\\
  \vb(s)&=0 & \text{ if } s\in[-\tau,-\varepsilon],\\
  \va(s)&=1+s/\varepsilon & \text{ if } s\in[-\varepsilon,0].
  \end{aligned}
  \right.
\end{equation}

Then, for all $t\in [-\tau,-\varepsilon]\cup[0,\tau-\varepsilon]$, we claim that 
\begin{equation*}
 d_V(t)=  I_0=I_1.
\end{equation*} 
 
We now prove this claim. It is clear that $d_V(t)=1$ for $t\in [-\tau,-\varepsilon]$. For $t\in(0,\tau-\varepsilon]$  we have that 
\begin{align*}
  \frac{d}{dt}\va(t)&=H^{(ab)}(t)\left(\vb(t-\tau)-\va(t)\right)\\
  & = -H^{(ab)}(t) \va(t) 
\end{align*}
and because $\va(0)=0$ we get $\va(t)=0$ for all $t\in[0,\tau-\varepsilon]$. And similarly $\vb(t)=1$ for all $t\in[0,\tau-\varepsilon]$.

Note that in  this case   equations \eqref{novalen} are not valid because $d_V$ is constant  during the time in $[0,\tau-\varepsilon]$. For clarity, we have simulated this example and plotted  in Figure \ref{figura1}. 

\begin{figure}
	\begin{minipage}{0.49\textwidth}
	   \includegraphics[width=\textwidth]{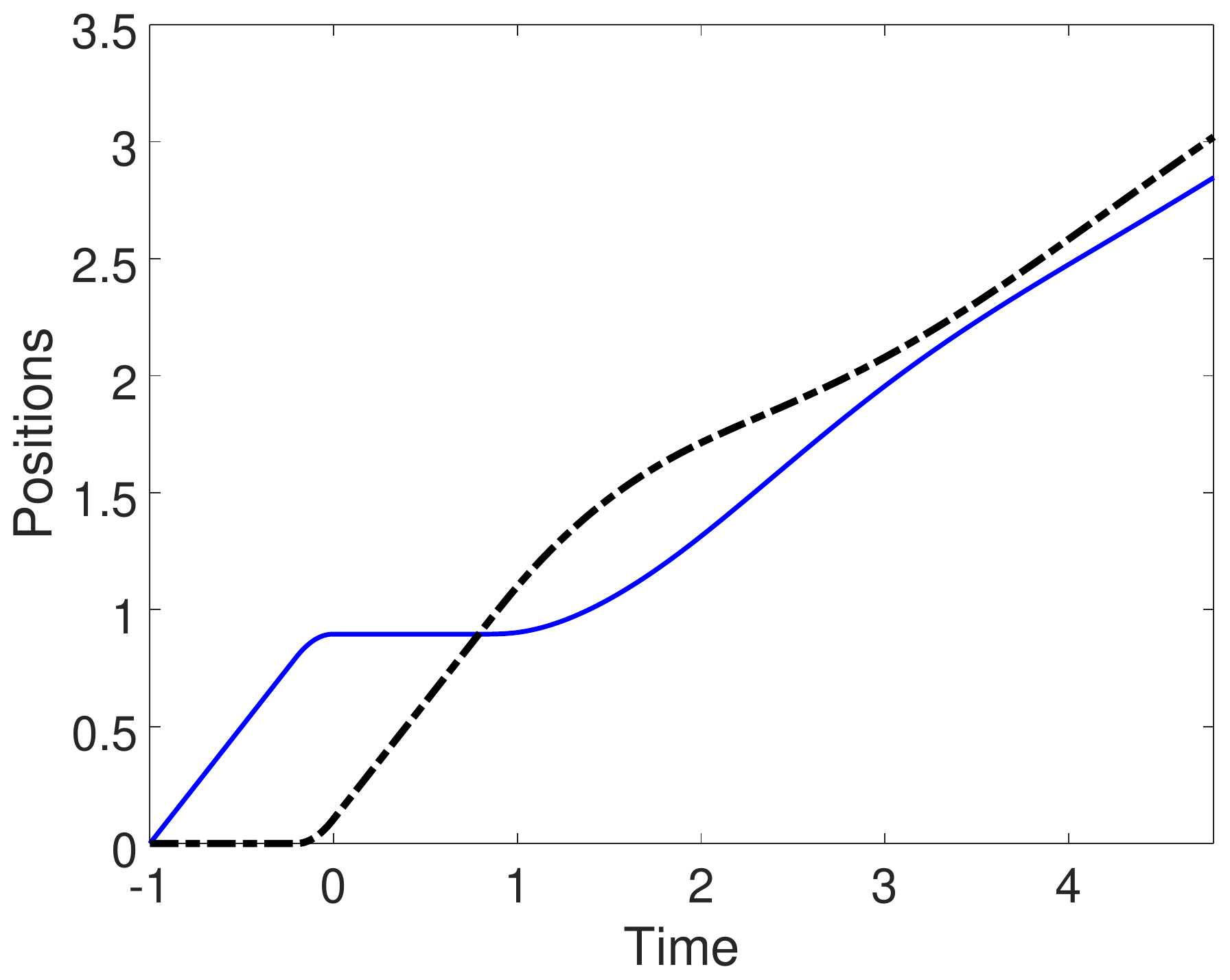}
	\end{minipage}
	\begin{minipage}{0.49\textwidth}
	   \includegraphics[width=\textwidth]{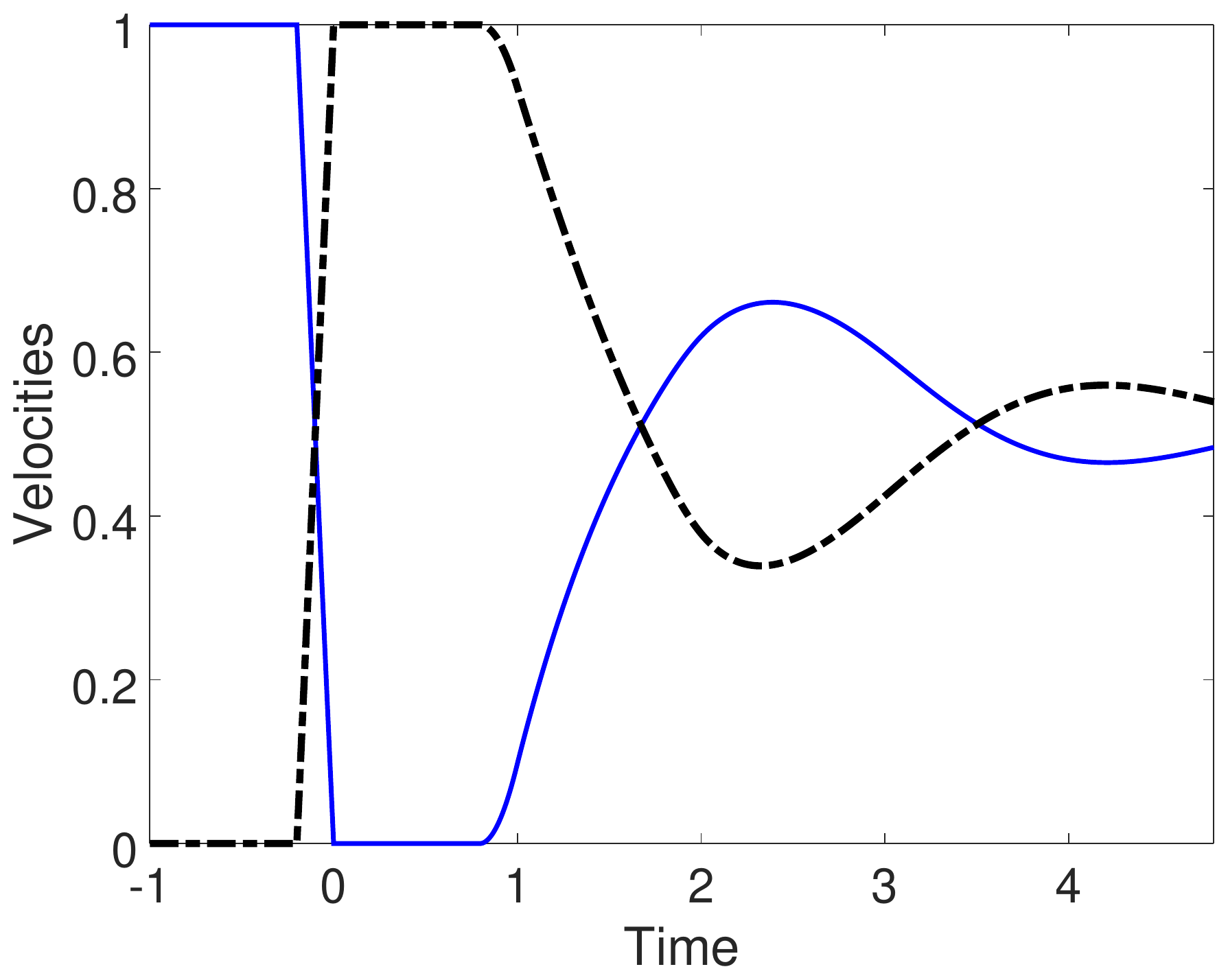}
	\end{minipage}	
	\caption{We simulate and plot Example 1 with $\tau=1$ and $\varepsilon= 0.2$.}\label{figura1}	
\end{figure}

\subsection{Example 2}\label{ejemplo2}

  Suppose that  $\tau=1$, $\psi(x)=1/\sqrt{1+x^2}$ and 
\begin{equation}
  \left\{\begin{aligned}
  \xa(s)&=1+(1+s)^2 & \text{ if } s\in[-1,0],\\
  \va(s)&=2(1+s) & \text{ if } s\in[-1,0],\\
  \xb(s)&=(1+s)^2 & \text{ if } s\in[-1,0],\\
  \vb(s)&=2(1+s) & \text{ if } s\in[-1,0].
  \end{aligned}
  \right.
\end{equation}
Note that   $d_V(s)=0$ for all $s\in[-1, 0]$ and yet we claim that 
\begin{equation}\label{desigualdadI0}
  \max_{s\in[0,\tau]}d_V(s)> \frac{I_0}{10}.
\end{equation}

We now prove this claim. Using the second inequality of  \eqref{gronwall} in the proof of Lemma \ref{cotasMm} for $\vu =1$, for each $t\in[0,1]$, we get
$$\va(t) \geq e^{-t}\va(0)+(1-e^{-t})\min_{c\in A}\min_{s\in [-\tau,0]}\vc(s)=2e^{-t}. $$
Then
\begin{equation}
  \va(t)\geq 2e^{-t}\geq 2e^{-1/2}\geq 2t
\end{equation}
if $t\in[0,1/2]$ and   the same holds for $\vb$. Therefore
$$\xa(t)-t^2\geq 2+\int_0^t2s\, ds-t^2=2$$
and we get 
\begin{align*}
 \frac{d}{dt} \va(t)&=  \frac{  2t-\va(t) }{\sqrt{1+(t^2-\xa(t))^2}} \geq \sqrt{2}\frac{ 2t-\va(t) }{ 1+  \xa(t)-t^2}
\end{align*}
where, as in \eqref{desigualdadbeta},  we are using that
$$\sqrt{2(1+X^2)}\geq  (1+|X|).$$
Now since 
$$\frac{d}{dt}(1+\xa(t)-t^2)=\va(t)-2t$$
we get 
\begin{align*}
  \va(t)&\geq \va(0)+\int_0^t (-\sqrt{2})\frac{ \va(s) -2s}{ 1+  \xa(s)-s^2}\, ds\\
  &=2-\sqrt{2}\left(\ln\left(1+  \xa(s)-s^2\right)\Big|_0^t\right)\\
  &=2-\sqrt{2} \ln\left(\frac{ 1+  \xa(t)-t^2}{3}\right)\\
  &=2-\sqrt{2} \ln\left( \frac{1+ 2+\int_0^t \va(s)\,ds-t^2}{3} \right)\\
  &=2 -\sqrt{2} \ln\left(1+\frac{  \int_0^t 2\,ds-t^2}{3} \right)\\
  &=2-\sqrt{2} \ln\left( 1+ \frac{ 2t-t^2}{3} \right)\\
  &\geq 2 -\sqrt{2} \ln\left( \frac{15}{4} \right)
\end{align*}
for $t\in[0,1/2].$ 

Additionally, applying \eqref{ejemplo2} to $\vb$ we have
$$\xb(t)-1-t^2\geq \xb(0)+\int_0^t\vb(s)\, ds-1-t^2\geq \int_0^t2s\, ds-t^2=0$$
and then
\begin{align*}
 \frac{d}{dt} \vb(t)&=  \frac{  2t-\vb(t) }{\sqrt{1+(1+t^2-\xb(t))^2}} \leq  \frac{ 2t-\vb(t) }{ 1+\xb(t)-1-t^2}
\end{align*}
because  $ 2t-\vb(t)\leq 0$ and $\sqrt{1+X^2}\leq 1+|X|.$ Similarly to the case  $\va$, we get 
\begin{align*}
  \vb(t)&\leq \va(0)+\int_0^t (-1)\frac{ \vb(s) -2s}{    \xb(s)-s^2}\, ds\\
  &=2- \ln\left(    \xb(t) -t^2 \right)\\
  &\leq 2- \ln\left(1+\int_0^t\vb(s)\, ds-t^2\right).
\end{align*}
Now assume 
$$\vb(t_0)=\min_{s\in[0,1/2]}\vb(s)$$
with $t_0\in[0,1/2]$. We claim that $\vb(t_0)\leq \alpha:=1.58$. We proceed by contradiction and suppose 
$$\vb(t)> \alpha$$
for each $t\in[0,1/2]$.  Then  
\begin{align*}
  \vb(1/2)&\leq 2- \ln\left(1+\int_0^{1/2}\alpha \, ds-\frac{1}{4}\right)\\
  &=2- \ln\left(1+ \frac{\alpha}{2}  -\frac{1}{4}\right)\\
  &<\alpha 
\end{align*}
and we obtain a contradiction. Finally
\begin{align*}
  d_V(t_0)&=\va(t_0)-\vb(t_0) \geq 2 -\sqrt{2} \ln\left( \frac{15}{4} \right)-\alpha> \frac{1}{10},
\end{align*}
and we get inequality \eqref{desigualdadI0}.
\begin{figure}
	\begin{minipage}{0.49\textwidth}
	  \includegraphics[width=\textwidth]{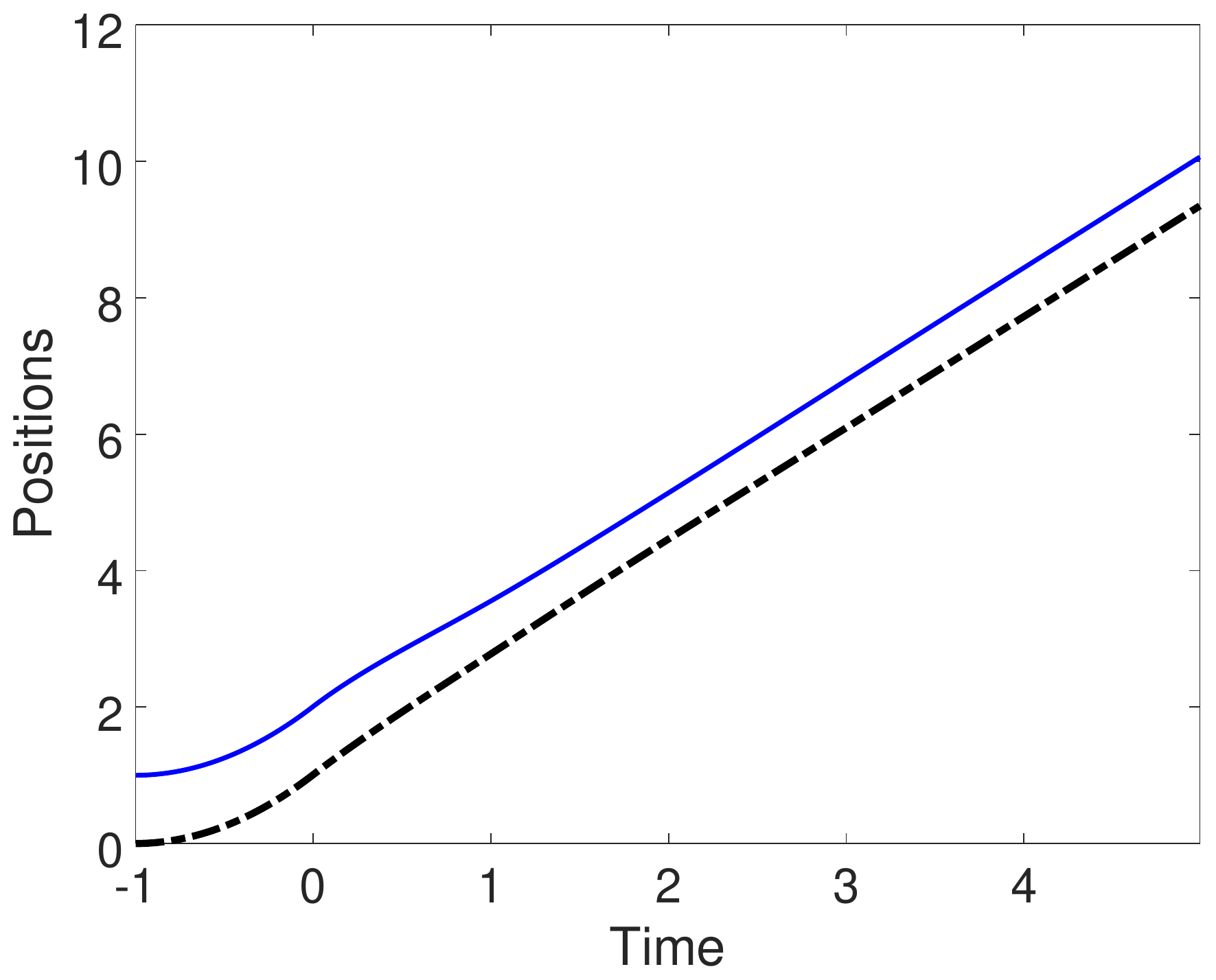}
	\end{minipage}
	\begin{minipage}{0.49\textwidth}
	   \includegraphics[width=\textwidth]{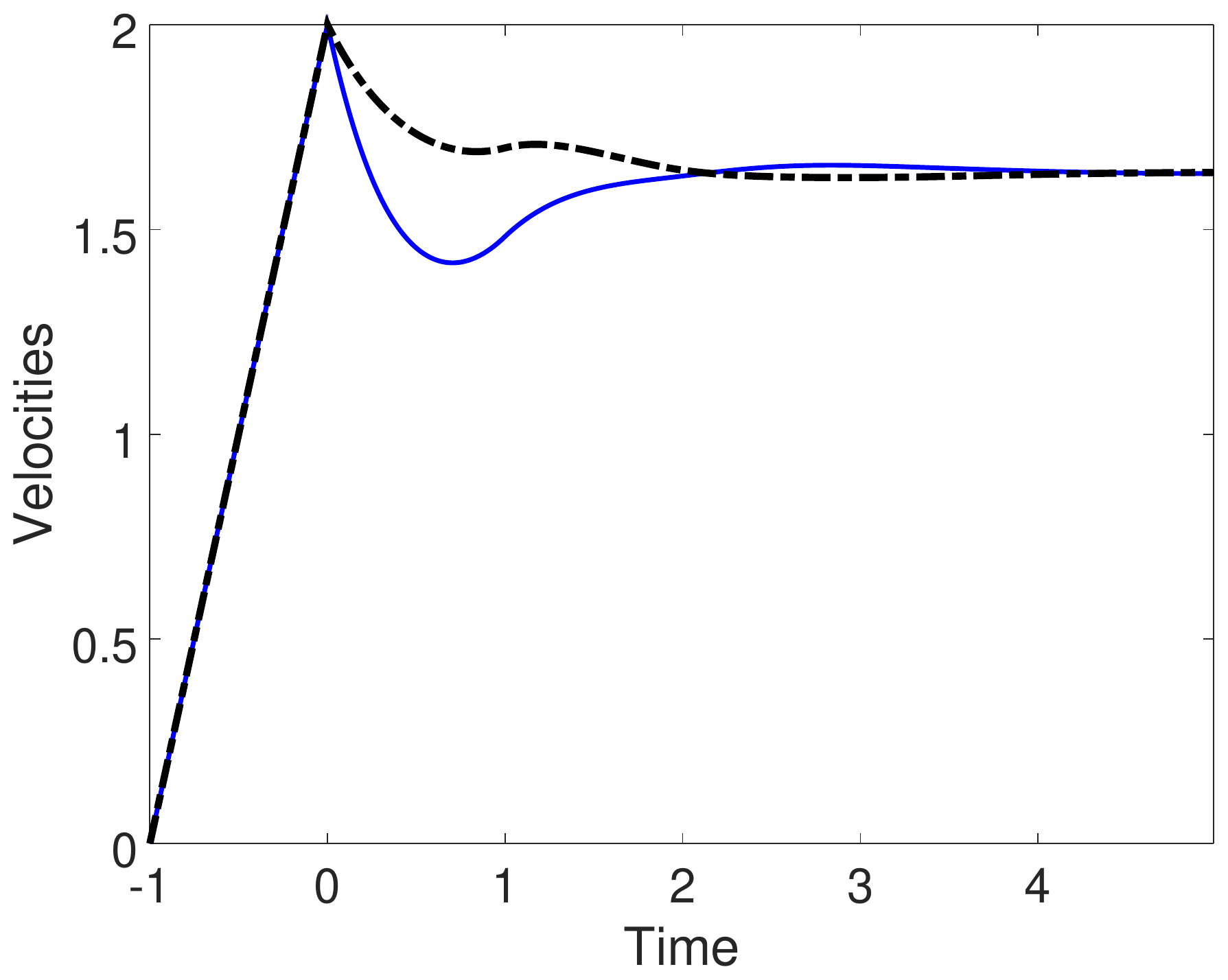}
	\end{minipage}	
	\caption{We simulate and plot Example   2}\label{figura2}	
\end{figure}

\addcontentsline{toc}{chapter}{Bibliograf\'ia}
\bibliography{biblio}
\bibliographystyle{plain}
\end{document}